\documentclass[11pt]{amsart}
\usepackage{amsmath}
\usepackage{amssymb}
\usepackage{amsfonts}
\usepackage{amsthm}
\usepackage{enumerate}
\usepackage[mathscr]{eucal}
\usepackage{verbatim}
\usepackage{amsthm}
\usepackage{amscd}

\newtheorem{theorem}{Theorem}[section]
\newtheorem{proposition}[theorem]{Proposition}

\newtheorem{corollary}[theorem]{Corollary}
\newtheorem{lemma}[theorem]{Lemma}

\theoremstyle{definition}

\numberwithin{equation}{section}

\def\flplp{{\Big(\sum_k\|f_{k}\|_p^p\Big)^{1/p}}}
\def\ellip{\text{{\rm ell}}}

\def\intslash{\rlap{\kern  .32em $\mspace {.5mu}\backslash$ }\int}
\def\qsl{{\rlap{\kern  .32em $\mspace {.5mu}\backslash$ }\int_{Q}}}

\def\vth{\vartheta}

\def\R{\mathbb R}

\def\emph#1{{\it #1 }}

\def\cf{{\it cf}}

\def\loc{{\text{\rm loc}}}

\def\coi{{C^\infty_0}}
\def\dist{{\text{\rm dist}}}

\def\supp{{\text{\rm supp }}}

\def\inn#1#2{\langle#1,#2\rangle}
\def\biginn#1#2{\big\langle#1,#2\big\rangle}

\def\lc{\lesssim}
\def\gc{\gtrsim}

\def\eps{\varepsilon}

\def\la{\lambda}

\def\fB{{\mathfrak {B}}}

\def\bbN{{\mathbb {N}}}

\def\bbR{{\mathbb {R}}}

\def\bbZ{{\mathbb {Z}}}

\def\cB{{\mathcal {B}}}

\def\cE{{\mathcal {E}}}
\def\cF{{\mathcal {F}}}

\def\cL{{\mathcal {L}}}

\def\cQ{{\mathcal {Q}}}
\def\cR{{\mathcal {R}}}

\def\cU{{\mathcal {U}}}

\renewcommand{\le}{\leqslant}
\renewcommand{\ge}{\geqslant}

\begin{document}

\address{Keith Rogers \\
Instituto de Ciencias Matematicas
CSIC-UAM-UC3M-UCM\\
28006 Madrid, Spain} \email{keith.rogers@uam.es}

\address{Andreas Seeger \\ Department of Mathematics \\ University of Wisconsin\\
480 Lincoln Drive\\
Madison, WI, 53706, USA} \email{seeger@math.wisc.edu}


\newcommand{\w}{\widehat{f}}
\newcommand{\B}{\mathbb{B}}


\author{Keith M. Rogers}
\author{Andreas Seeger}

%\date{\today}
\title[Maximal and smoothing estimates  for Schr\"odinger equations]{Endpoint maximal and smoothing estimates for Schr\"odinger equations}
\thanks{The first author was supported by MEC  project
 MTM2007-60952 and UAM-CM project CCG07-UAM/ESP-1664. The second author was supported in part by  NSF grant DMS  0652890.
}

\begin{abstract} For $\alpha>1$ we consider the initial value problem for the dispersive equation
%the dispersive equation
 $i\partial_tu +(-\Delta)^{\alpha/2} u= 0$.
%with $u(\cdot,0)=f$.
We prove an endpoint $L^p$ inequality for the maximal function
 $\sup_{t\in[0,1]}|u(\cdot,t)|$ with initial values in
$L^p$-Sobolev spaces, for
 $p\in(2+4/(d+1),\infty)$. This
 strengthens the fixed time estimates
 due to Fefferman and Stein, and Miyachi. As an  essential tool we establish sharp  $L^p$ space-time estimates (local in time) for the same
 range of $p$.
\end{abstract}

\maketitle

\section{Introduction}\label{intro}

For 
%$\alpha\in(0,1)\cup(1,\infty)$, 
$\alpha>1$ we
 consider $L^p$ estimates for solutions  to the initial value problem
$$
\mbox{}\left\lbrace
\begin{array}{l}
i\partial_tu +(-\Delta)^{\alpha/2} u= 0\\
u(\,\cdot\,,0)=f.
\end{array}
\right.
$$
The case $\alpha=2$ corresponds to the Schr\"odinger equation. We
will not consider $\alpha=1$ which corresponds to the wave equation
and exhibits different mathematical features.

When $f$ is a Schwartz function, the solution can be written as
$u(x,t) =U_t^\alpha f(x)$, where
\begin{equation}\label{form}
\widehat{U_t^\alpha f} (\xi) = e^{it |\xi|^\alpha}\widehat
f(\xi)\end{equation} with
$\widehat f(\xi)=\int f(y)e^{-i\inn{y}{\xi} }dy$  as the definition of the
 Fourier transform.
 The sharp endpoint $L^p$-Sobolev bounds for fixed $t$ are
due to Fefferman and Stein \cite{fest} and Miyachi~\cite{mi}. Their
result states that for any compact time interval $I$
 and any $p\in(1,\infty)$,
$$
\sup_{t\in I}\big\|U_t^\alpha f\big\|_{L^p(\R^d)}\le
C_{I\!,p,\alpha} \|f\|_{L^p_{\beta}(\R^d)},\quad
\frac{\beta}{\alpha}= d\Big|\frac12-\frac1p\Big|;
$$
this is sharp with respect to the regularity index $\beta$ and can
also   be deduced from certain endpoint
 versions of the
H\"ormander multiplier theorem (\cite{basa}, \cite{se}).

We strengthen the fixed time estimates as follows.
\begin{theorem}\label{pop1} Let $p\in (2+\frac{4}{d+1},\infty)$ and $\alpha>1$.
Then, for any compact time interval $I$,
\begin{equation}\label{maxestimate}
\big\|\sup_{t\in I}|U_t^\alpha f|\,\big\|_{L^p(\R^d)}\le
C_{I\!,p,\alpha} \|f\|_{L^p_{\beta}(\R^d)},\quad
\frac{\beta}{\alpha}= d\Big(\frac12-\frac1p\Big).
\end{equation}
\end{theorem}
This implies pointwise convergence results; indeed we shall  prove a
little more, namely  if $\chi\in C^\infty_c(\bbR)$
then the function
$t\mapsto \chi(t) U_{t}^\alpha f(x)$  belongs to the Besov space
$B^p_{1/p,1}(\bbR)$, for almost every $x\in \bbR^d$.
In particular these functions are continuous (for
almost every $x$) and therefore this implies
almost everywhere convergence to the initial datum as $t\to 0$.

Our maximal function result is closely related to certain space-time
estimates which improve the regularity index. The first such bounds
are due to
Constantin and Saut \cite{cosaut}, Sj\"olin \cite{sj}, and Vega
\cite{V} who showed  that  better $L^2$ regularity properties hold locally
when $\alpha\in(1,\infty)$; namely, if $f\in
L^2_{-(\alpha-1)/2}(\R^d)$ then $u\in L^2_{\loc}(\R^{d+1})$.
However, it is not possible to replace the $L^2$-norms over compact
sets by $L^2$-norms which are global in space. This is known as the
{\it local smoothing} phenomenon.  For functions in
 $L^2$-Sobolev spaces
the various local and global
 problems
for  smoothing and for maximal operators have received a lot of
attention, starting with \cite{camax}. We do not have a contribution
to the $L^2$-Sobolev problems  but rather  consider
corresponding questions with   initial data in  $L^p$-Sobolev spaces
for $p>2$, with $p$ not close to $2$.

In \cite{ro1} the first author considered $L^p$ regularity estimates
which are global in space but involve an integration over a compact
time interval $I$,
\begin{equation} \label{sm}
\Big(\int_I\|U_t^\alpha f\|^p_{p}\,dt\Big)^{1/p}\le C_I
\|f\|_{L^p_\beta(\bbR^d)}.
\end{equation}
%here $I$ is a compact time interval.
This question was motivated by
the  similar (although deeper) question for the wave equation (\cf.
\cite{so}, \cite{wo1}). In \cite{ro1}, it was proven that (\ref{sm})
holds for $\alpha=2$ when $p> 2+4/(d+1)$ with
$\beta/2>d(1/2-1/p)-1/p$. We remark that
smoothing results of this type could also be deduced from
square-function estimates related to Bochner-Riesz multipliers such
as in \cite{ca-duke}, \cite{ch}, \cite{se0} and \cite{leese}
however these arguments do not apply when
 $d=1$, and in dimensions $d\ge 2$
they  are currently limited to the smaller range $p>2+4/d$.

The $L^p$ smoothing result in \cite{ro1} was obtained from an
$L^p\to L^p$ estimate for the adjoint Fourier restriction  (or \lq
extension') operator associated to the paraboloid, and the range $p>
2+\frac{4}{d+1}$ corresponds to the known range of $L^q\to L^p$
bounds for  the extension operator; see \cite{fe-acta}, \cite{hoer}
and \cite{zyg}  for the sharp bounds when  $d=1$, and \cite{ta} for
the best known partial results for $d\ge 2$. The reduction in
\cite{ro1} to the extension estimate used the explicit
formula
$$e^{it\Delta}f(x)= \frac{1}{(4\pi it)^{d/2}} \int e^{i|x-y|^2/4t} f(y) dy$$
together with a
\lq completing of the square' trick; see \cite{ca} for a similar
argument.
Unfortunately this  reasoning is  not  available when
$\alpha\neq 2$.

We generalize to all $\alpha>1$, and establish the endpoint
regularity result.

\begin{theorem}\label{the}
 Let $p\in(
2+\frac{4}{d+1},\infty)$ and $\alpha>1$. Then, for any compact time
interval $I$,
\begin{equation*}
\Big(\int_I\|U_t^\alpha f\|^p_{p}\,dt\Big)^{1/p}\le C_{I\!,p,\alpha}
\|f\|_{L^p_\beta(\bbR^d)},\quad \frac \beta\alpha= d \Big(\frac 12-
\frac 1p\Big)-\frac1p.
\end{equation*}
\end{theorem}
In Theorem \ref{trliz} below we formulate a slightly improved version
of this result which can also be used to prove Theorem \ref{pop1}.
We remark that for $d=1$ our arguments also give the analogous results for the range $0<\alpha<1$.

We mention an application in one spatial dimension where we obtain
 sharp  estimates for the initial value problem for the Airy equation
\begin{equation}\label{airy}
%\mbox{}\left\lbrace
%\begin{array}{l}
u_t+ u_{xxx} = 0.
%\\u(\,\cdot\,,0)=f.
%\end{array}\right.
\end{equation}
For $f:= u(\cdot,0)$ a Schwartz function, we can write
$u(\cdot, t)=
U^{3}_{t}P_+ f + U^{3}_{-t} P_-f$, where $P_+$ and $P_-$ are the
projection operators with Fourier multipliers $\chi_{(0,\infty)}$
and $\chi_{(-\infty,0)}$, respectively. Thus,
for initial values in $L^p_\beta$  the solution of \eqref{airy} satisfies the sharp bound
$$
\|u\|_{L^p(\bbR\times[-T,T])}\le C_{T} \|u(\cdot,0)\|_{L^p_\beta(\R)},
\quad\beta= \frac{3(p-4)}{2p}, \quad 4<p<\infty,
$$
and if $u(\cdot,0)\in L^p_\eps(\bbR)$ for any
$\eps>0$ with $2<p\le 4$, then  $u\in
L^p(\bbR\times[-T,T])$.

\medskip

The proofs will be based on the bilinear adjoint restriction theorem
for elliptic surfaces due to Tao \cite{ta}. In \S\ref{ell}, having
discussed the necessary conditions in \S\ref{necessary},  we combine
Tao's theorem with
 a variation of a localization technique employed in \cite{fe}  to
prove  $L^p$ estimates for some oscillatory integrals with elliptic
phases; this yields the smoothing estimate for functions which are
frequency supported in an annulus. In \S\ref{localtoglobal}, we
extend to the general case
 by decomposing the Fefferman-Stein sharp
function; here we use a variant of an argument in~\cite{se}.

\medskip{\it Notation.}
Throughout, $c$ and $C$ will denote positive constants that may
depend on the dimensions, exponents or indices of the Sobolev
spaces, or the parameter $\alpha$, but never on the functions.
 Such constants are called admissible and
their values may change from line to line. We shall mostly  use the
notation $A\lc B$ if $A\le CB$ for an admissible constant~$C$. We
may sometimes indicate the dependence on a specific  parameter $c$
by using the notation~$\lc_c$. We write $A\approx B$ if $A\lc B$ and
$B\lc A$.

\section{Necessary conditions}\label{necessary}

Let $\theta$  be a nonnegative  and smooth function supported in
$\{2^{-1}<|\xi|<2\}$
 and equal to 1 in $\{2^{-1/2}<|\xi|<2^{1/2}\}$.
For large $\la$, we consider initial data $f_\la$ defined by
$
\widehat{f}_\la(\xi)=
e^{-i|\xi|^\alpha}\theta(\la^{-1}\xi)
$
and note that, by a change of variables,
$$
f_\la(x)=\left(\frac{\la}{2\pi}\right)^d\int\theta(\xi)
e^{i(\inn{\la x}{\xi}-\la^\alpha|\xi|^\alpha)}d\xi.
$$
Thus $|f_\la(x)|\lc \la^{d-\frac{d\alpha}{2}}$, by  the method of
stationary phase (keeping in mind that $\alpha\neq 1$). On the other
hand, when $|x|\gg  \la^{\alpha-1}$, by repeated integration by
parts, there exists constants $C_N$ such that $|f_\la(x)|\le C_N
(|x|\la^{1-\alpha})^{-N}$ for all $N\in\bbN$. Combining the two
bounds, we see that
\begin{equation*}
\|f_\la\|_{L^p_\beta(\bbR^d)}
\approx \la^\beta \|f_\la\|_{L^p(\bbR^d)} \lc
\la^{{d-\frac{d\alpha}{2}}+\frac{d(\alpha-1)}{p}+\beta}.
\end{equation*}

Next  we consider
$U_t^\alpha
f_\la$
and compute
\begin{align*}
|U_t^\alpha f_\la(x)|&=\Big|
\left(\frac{\la}{2\pi}\right)^{d}\int_{\bbR^d}\theta(\xi) e^{
i(\inn{\la x}{\xi}+ \la^\alpha(t-1)|\xi|^\alpha)}d\xi\Big|,
\end{align*}
so when $|x|\le(10\la)^{-1} $ and $|t-1|\le (10\la^\alpha)^{-1},$ we
have $|U_t^\alpha f_\la(x)|\ge c\la^{d}$ for some positive constant
$c$. Thus,
$$
\Big(\int_{1-(10\la^\alpha)^{-1}}^1 \|U_t^\alpha f_\la\|_p^p
\,dt\Big)^{1/p} \ge C \la^{d-\frac{d+\alpha}{p}}.
$$
Comparing this with the upper bound for
$\|f_\la\|_{L^p_\beta(\bbR^d)}$, and letting $\la\to \infty$, we see
that $\beta/\alpha\ge d(1/2-1/p)-1/p$  is a necessary condition for
(\ref{sm}) to hold when $\alpha\neq 1$.

Note that alternatively one can argue that by Sobolev embedding
 any improvement in the smoothing would give  a better fixed time
estimate than the sharp known bounds in \cite{fest}, \cite{mi}, which is impossible.

The range $p>2+4/(d+1)$ for the smoothing estimate in Theorem \ref{the}
is sharp
for  $d=1$, and for $d\ge 2$  it is conceivable that it holds for
$p>2+2/d$, see \cite{ro1}.

For Theorem~\ref{pop1} however
our range may not be sharp even in one dimension.
We can say that
   the maximal estimate \eqref{maxestimate}
 cannot hold when $p<2+1/d$. This follows from the
necessary condition
 $\beta/\alpha\ge 1/2p$ which we now show, modifying a
 calculation in \cite{dake}.

Let $\chi$ be a nonnegative and smooth function supported in
$(-\eps, \eps)$ where $\eps$ will be small depending  only on
$\alpha$. Let $e_1=(1,0,\dots,0)$ and define
%the function $g_N$ by its Fourier transform
$$g_\la(x)= \frac{1}{(2\pi)^{d}}\int \chi(\la^{\frac{\alpha-2}{2}}|\xi+\la e_1|) e^{i\inn x\xi} d\xi.$$
Then immediately
\begin{equation*}
%\label{uppered}
\|g_\la\|_{L^p_\beta} \lc \la^{\beta + \frac{d(\alpha-2)}{2} (\frac
1p -1)}.\end{equation*} Now
\begin{align*}
U^\alpha_tg_\la(x)&= \frac{1}{(2\pi)^{d}}\int
\chi(\la^{\frac{\alpha-2}{2}}|\xi+\la e_1|) e^{i(\inn
x\xi+t|\xi|^\alpha)} d\xi
\\
&=  \frac{1}{(2\pi)^{d}} \int \chi(\la^{\frac{\alpha-2}{2}}|h|)
e^{i\phi_\la(x,t,h)} dh
\end{align*}
where $\phi_\la(x,t,h)= t \la^\alpha |-e_1+h/\la|^\alpha +
\inn{x}{-\la e_1+h}$. A Taylor expansion gives for $|h|\ll \la$
$$\phi_\la(x,t,h)=t\la^\alpha-x_1 \la  + \inn{x- t\alpha \la^{\alpha-1} e_1}{h} +O(\la^{\alpha-2}h^2)$$
where the implicit constants in the error term depend on $\alpha$. The error term in the phase is $\ll 1$  on the support of the cutoff function
(provided that $\eps$ is sufficiently small).

Let $0<c\ll \alpha$  and let $R$ be the rectangle where $0\le x_1\le
c \la^{\alpha-1}$, and $|x_i|\le \la^{(\alpha-2)/2}$ for
$i=2,\dots,d$. We define $t(x)=\alpha^{-1}\la^{1-\alpha}x_1$ for
$x\in R$ so that $t(x)\in [0,1]$ for  $x\in R$,
 and for $x\notin R$ we may choose any (measurable) $t(x)\in [0,1]$. Then for $x\in R$, we have
$|U^\alpha_{t(x)}g_\la(x)|\ge c_0 \la^{-d(\alpha-2)/2}$ and thus
$$\big\|\sup_{0\le t\le 1}|U^\alpha_t g_\la|\, \big\|_p\ge
\|U^\alpha_{t(\cdot)}g_\la\|_p\gc
\la^{\frac{\alpha-1}{p}+\frac{(\alpha-2)(d-1)}{2p}-\frac{(\alpha-2)d}{2}}.$$
Comparing with the upper bound for $\|g_\la\|_{L^p_\beta}$
   leads to the condition $\beta/\alpha\ge 1/2p$.

\section{$L^p$ estimates for oscillatory integrals with elliptic phases}
\label{ell}

In the sequel, we will rescale inequalities for $U^\alpha_t$ when
acting on functions with compact frequency support. This process
will give rise to the operator $S$ defined by
\begin{equation}  \label{Sdef}
Sf(x,t) \equiv S^\phi_\chi f(x,t) = \frac{1}{(2\pi)^{d}} \int
%e^{i\inn{x}{\xi}+t\phi(\xi)}
\chi(\xi)e^{it\phi(\xi)} \widehat f(\xi) e^{i\inn{x}{\xi}}d\xi
\end{equation}
where $\chi\in \coi(\cU)$ and $\phi$ is {\it elliptic}; here
 a $C^\infty $ function
$\phi$ on an open set $\cU $ in $\bbR^d$  is called  elliptic  if
for every $\xi\in \cU$ the Hessian $\phi''$ is positive definite.

We ask for $L^p(\bbR^d)\to L^p(\bbR^d\times [0,\la])$ bounds for
$S$. Note that  for $|t|\le 1$ and $\chi\in \coi$ the function
$\chi e^{it\phi}$ is a Fourier multiplier of $L^p$, $1\le
p\le\infty$, and consequently the  question is only nontrivial for
large $\la$.

\begin{proposition} \label{ellipticthm}
Let $p>2+\frac{4}{d+1}$, $\chi\in \coi(\cU)$,  and let $\phi$ be an
 elliptic
phase  on $\cU$. Then
$$\|Sf\|_{L^p(\bbR^d\times [-\la,\la])}\lc \la^{d(1/2-1/p)}
\|f\|_{L^p(\bbR^d)}.$$
\end{proposition}

The key ingredient will be Tao's bilinear estimate for the adjoint
restriction operator \cite{ta} which applies to phases which are
small perturbations of $|\xi|^2/2$. We need to formulate more
specific assumptions on the phases allowed and follow \cite{tavave}.
Let $N\ge 10 d$. We say $\phi :[-2,2]^d\to \bbR$ is a phase  of the
class $\Phi(N,A)$ if $|\partial^{\alpha_j}_{x_j} \phi(x)|\le A$ for
all $x\in [-2,2]^d$ and all $|\alpha_j|\le N$, where $j=1,\ldots,d$.
To add an ellipticity condition we say that $\phi$ is of class
$\Phi_\ellip(\eps,  N, A)$ if $\phi(0)=\nabla\phi(0)=0$, and if for
all $x\in [-2,2]^d$ the eigenvalues of the Hessian $\phi''(x)$ lie
in $[1-\eps, 1+\eps]$.

We define the adjoint restriction operator $\cE\equiv \cE^\phi$ by
$$
\cE h(x,t)= \int_{[-2,2]^d}  e^{i(\inn{x}{\xi} +t\phi(\xi))} h(\xi)
d\xi.
$$
so that
  $Sf =(2\pi)^{-d}\cE \widehat f$, where $\cU=(-2,2)^d$. Now
Tao's theorem can be stated as follows. Suppose $p>2+\frac{4}{d+1}$.
Then
 there exists an $N$ (depending on $d$ and $p$) and for $A\ge 1$ there exists
 $\eps =\eps (A,N,d,p)>0$ so that the following holds for $\phi \in \Phi(\eps,  N, A)$: For all
pairs of $L^2$ functions $h_1$, $h_2$ so that $\dist(\supp(h_1),
\supp (h_2))\ge c>0$ the inequality
\begin{equation}\label{terry}
\big\| \cE h_1\cE h_2\big\|_{p/2} \lc_c \|h_1\|_2 \|h_2\|_2, \quad
p>2+\frac{4}{d+1},
\end{equation}
holds. In what follows we fix $N$, $A$ and $\eps$ for which Tao's
theorem applies. The constants may all depend on these parameters.

\begin{lemma}\label{1-separated}
 Let $p>2+\frac{4}{d+1}$, let $B_1$, $B_2\subset[-1,1]^d$ be balls so that $\dist (B_1, B_2) \ge c$, and let $\phi\in
\Phi_\ellip(\eps,N,A)$. Then for $f$, $g$ with $\supp \widehat
f\subset B_1$, $\supp \widehat f\subset B_2$,
\begin{equation*}
\big\|Sf\, Sg\big\|_{L^{p/2}(\bbR^d\times[0,\la])} \lc_{c,p} \la^{d
(1-2/p)} \|f\|_{L^p(\R^d)} \|g\|_{L^p(\R^d)}\,.
\end{equation*}
\end{lemma}
\begin{proof} Let $C_0=10 (1+ \max_{\xi\in [-2,2]^d} |\nabla \phi(\xi)|)$, and let $\eta_1, \eta_2\in \coi$ be  supported in $(-2,2)^d$ so that
$\eta_1(\xi)=1$ on $B_1$ and $\eta_2(\xi_2)=1$ on $B_2$. Moreover
assume that $\eta_1$ and $\eta_2$ are supported on slightly larger
concentric balls $\widetilde B_1$, $\widetilde B_2$ with the
property that $\dist (\widetilde B_1,\widetilde B_2)\ge c/2$. We
also set
$$P_i f= \cF^{-1}[\eta_i \widehat f],\quad i=1,2.$$
Let $K^i_t= \cF^{-1}[e^{it\phi}\eta_i\chi ]$, for $i=1,2$, so that
$$S_i f(x,t):= SP_i f(x,t)=K^i_t*f(x).$$ Then $Sf\, Sg = S_1 f \,S_2 g$.
We first note that for all $t\in [-\la,\la]$
\begin{equation}\label{error}|K_t^i(x)|\lc |x|^{-N}, \quad \text{ if } |x|\ge C_0\la
\end{equation} This follows by a straightforward $N$-fold integration by
parts, which uses the inequality $|\nabla_\xi
(\inn{x}{\xi}+t\phi(\xi))| \ge |x|/2$ if $|x|\ge C_0\la$, $|t|\le
\la$.

Now let $\cQ(\la)$ be a tiling of $\bbR^d$ by cubes of sidelength
$\la$, and for each $Q\in \cQ(\la)$ let $Q_*$ denote the enlarged
cube with sidelength $2C_0\la$, with the same center as $Q$.  For
each cube we split each function into a part supported in $Q_*$ and
a part supported in its complement. Thus we can  write
$$
\big\|Sf\, Sg\big\|_{L^{p/2}(\bbR^d\times[0,\la])}^{p/2} = I +II
+III +IV
$$ where
\begin{align*}
I&=\sum_{Q\in \cQ(\la)} \big\|S_1[f\chi_{Q_*}]\,
S_2[g\chi_{Q_*}]\big\|_{L^{p/2}(Q\times[0,\la])}^{p/2}\,,
\\
II&=\sum_{Q\in \cQ(\la)} \big\|S_1[f\chi_{Q_*}]\,
S_2[g\chi_{\bbR^d\setminus Q_*}]
\big\|_{L^{p/2}(Q\times[0,\la])}^{p/2}\,,
\\
III&=\sum_{Q\in \cQ(\la)} \big\|S_1[f \chi_{\bbR^d\setminus Q_*}]\,
S_2[g\chi_{Q_*}] \big\|_{L^{p/2}(Q\times[0,\la])}^{p/2}\,,
\\
IV&=\sum_{Q\in \cQ(\la)} \big\|S_1[f\chi_{\bbR^d\setminus Q_*}]\,
S_2[g\chi_{\bbR^d\setminus
Q_*}]\big\|_{L^{p/2}(Q\times[0,\la])}^{p/2}\,.
\end{align*}
The first term gives the main contribution and is
 estimated using Tao's theorem, i.e. \eqref{terry}.
One obtains,
\begin{align*}
|I|&\le \sum_{Q\in \cQ(\la)} \big\|S\,P_1[f\chi_{Q_*}]\,
 S\,P_2[g\chi_{Q_*}]\big\|_{L^{p/2}
(\bbR^d\times\bbR)}^{p/2}
 \lc_c \sum_Q
\big\| P_1[f\chi_{Q_*}] \big\|_2^{p/2}  \big\| P_2
[g\chi_{Q_*}]\big\|_2^{p/2}
\\&
\lc \sum_Q \big\| f\chi_{Q_*} \big\|_2^{p/2}  \big\|g\chi_{Q_*}
\big\|_2^{p/2} \lc \Big(\sum_Q \|f\chi_{Q_*}\|_2^p\Big)^{1/2}
\Big(\sum_Q \|g\chi_{Q_*}\|_2^p\Big)^{1/2}.
\end{align*}
By H\"older's inequality,
\begin{equation*}
\Big(\sum_Q \|f\chi_{Q_*}\|_2^p\Big)^{1/p} \lc \Big(\sum_Q
|Q_*|^{p/2-1} \|f\chi_{Q_*}\|_p^p\Big)^{1/p} \lc \la^{d(1/2-1/p)}
\|f\|_p,
\end{equation*}
and we have the same estimate for $g$. Thus $I^{2/p} \lc_c
\la^{d(1-2/p)} \|f\|_p \|g\|_p$ which is the desired bound for the
main term.

The corresponding estimates for $II$, $III$, $IV$ are
straightforward
 as we
use  \eqref{error} for the terms supported in $\bbR^d\setminus Q_*$.
We examine $II$ and begin with
%\footnote{\texttt{change} $\R^d$ to $Q$}
\begin{align}
|II|&\le\sum_{Q\in \cQ(\la)}
\big\|S_1[f\chi_{Q_*}]\big\|_{L^{p}(Q\times[0,\la])}^{p/2} \big\|
S_2[g\chi_{\bbR^d\setminus Q_*}]
\big\|_{L^{p}(Q\times[0,\la])}^{p/2} \notag
\\
&\le\Big(\sum_{Q\in \cQ(\la)}
\big\|S_1[f\chi_{Q_*}]\big\|_{L^{p}(Q\times[0,\la])}^{p} \Big)^{1/2}
\Big(\sum_{Q\in \cQ(\la)}\big\| S_2[g\chi_{\bbR^d\setminus Q_*}]
\big\|_{L^{p}(Q\times[0,\la])}^{p} \Big)^{1/2}. \label{product}
\end{align}
We use the trivial  bound $\|S_1 f(\cdot, t)\|_p\lc (1+|t|)^d
\|f\|_p$ for $f$ replaced with $f\chi_{Q_*}$, so that the first
factor in \eqref{product} is bounded by $(C\la^{d+1}
\|f\|_p)^{p/2}$. By \eqref{error} we get
%\footnote{\texttt{change}$2=d+1$, we are integrating in d variables not 1}
\begin{multline*}
\Big(\sum_{Q\in \cQ(\la)}\big\| S_2[g\chi_{\bbR^d\setminus Q_*}]
\big\|_{L^{p}(Q\times[0,\la])}^{p}  \Big)^{1/p}
\\ \lc\Big(\int_{-\la}^\la \int_{x\in \bbR^d}\Big[
\int_{|z|\ge \la} |z|^{-N} |g(x-z)| dz\Big]^p dx dt\Big)^{1/p} \lc
\la ^{d+1-N} \|g\|_p\,.
\end{multline*}
Hence $|II |^{2/p}\lc_c \la^{2(d+1)-N} \|f\|_p\|g\|_p.$ As $N\ge
10d$ this estimate is negligible. Because of  symmetry  $III$ is
estimated by the same term. For the estimation of $IV$ we proceed in
the same way but use \eqref{error} for both terms, the result is the
(again negligible)  bound $|IV |^{2/p}\lc  \la^{2(d+1-N)}
\|f\|_p\|g\|_p.$
 \end{proof}

We now formulate an analogous result  for functions with smaller
frequency support and smaller separation.

\begin{lemma}\label{smallballsii}
Let $p> 2+\frac{4}{d+1}$ and  $\la^{1/2}\ge 2^{j}\ge1.$ Let $Q_1$,
$Q_2\subset[-1,1]^d$ be cubes of side $2^{j}\la^{-1/2}$, so that
$\dist (Q_1, Q_2) \ge c  2^{j}\la^{-1/2}$ and let
$\phi\in\Phi_\ellip(\eps,N,A)$.
%\footnote{\texttt{change} $\Phi$ to $\Phi_\ellip$. Also in the next
%page.}
Then for all $f$ and $g$ such
that $\,\supp \widehat f\subset Q_1$, $\supp \widehat f\subset Q_2$,
\begin{equation*}
\big\|Sf\, Sg\big\|_{L^{p/2}(\bbR^d\times[0,\la])} \lc_c 2^{4j
(\frac d2-\frac{d+1}{p})} \la^{\frac 2p} \|f\|_{L^p(\R^d)}
\|g\|_{L^p(\R^d)}\,.
\end{equation*}
\end{lemma}

\begin{proof}
 By finite partitions and the triangle inequality, we may suppose that $Q_1$ and $Q_2$ are balls of radius $2^{j}\la^{-1/2}$. We reduce matters to the
statement in Lemma  \ref{1-separated} by scaling. Let $\xi_0$ be the
midpoint of the interval connecting the center of the balls. We
change variables $\xi=\xi_0+ \delta\eta$ where $\delta=
2^{j}\la^{-1/2}$. Then a short computation shows
that
%\footnote{\texttt{change} $x_0$ to $\xi_0$}
\begin{equation*}
S^\phi \!f(x,t) = e^{i(\inn{x}{\xi_0}+t\phi(\xi_0))} S^\psi \!f_*
(\delta (x+t\nabla \phi(\xi_0)), \delta^2 t)\quad \text{where }
f_*(y) = f(\delta^{-1}y)e^{i\delta^{-1}\inn{y}{\xi_0}},
\end{equation*}
and the phase $\psi$ is given by
$$\psi(\eta)= \frac{1}{2} \int_0^1
 \inn{\phi''(\xi_0+s\delta\eta) \eta}{\eta} ds.$$
The same consideration is applied to $S^\phi \!g$. Note that $\psi$
is elliptic (with estimates uniform in $\xi_0$ and $\delta$)
and  the frequency supports of $f_*$ and $g_*$ are now
separated, independently of $\delta$, $j$ and $\la$. Thus we can
apply Lemma~\ref{1-separated} to obtain
\begin{align*}
\|S^\phi \! f\,S^\phi\! g\|_{L^{p/2} (\bbR^d\times [0,\la])}&
=\delta^{-(d+2)/(p/2)} \|S^\psi \!f_*\,S^\psi \!g_*\|_{L^{p/2}
(\bbR^d\times [0,\la \delta^2])}
\\&\lc \delta^{-(2d+4)/p} (\la \delta^2)^{d(1-2/p)} \|f_*\|_p
\|g_*\|_p\\ &\lc \delta^{2d-4(d+1)/p} \la^{d(1-2/p)}\|f\|_p\|g\|_p.
\end{align*}
As $\delta=2^{j}\la^{-1/2}$ the assertion follows.
\end{proof}

We will also require the following lemma for when we have no
frequency separation.

\begin{lemma}\label{smallballsi}
Let $p\ge 1$, let $Q\subset[-1,1]^d$ be a cube of side $\la^{-1/2}$,
and let $\phi\in \Phi(N,A)$. Then for all $f$ such that $\,\supp
\widehat f\subset Q$,
\begin{equation*}
\|S f(\cdot, t)\|_{L^p(\R^d)} \lc \|f\|_{L^p(\R^d)}, \quad |t|\le
\la.
\end{equation*}
\end{lemma}

\begin{proof}
Let $\xi_B$ be the center of the cube $Q$, and let $\chi\in \coi$ so
that $\chi(\xi)=1$ for $|\xi|\le \sqrt{d}$.
 It suffices to show that
$\chi(\la^{1/2}(\xi-\xi_B)) e^{it\phi(\xi)}$ is a Fourier multiplier
of $L^p$ for all $|t|\le \la$, with bounds uniform in $t$. By
modulation, translation and dilation invariance of the multiplier
norm it suffices to check that $h(\cdot,t)$ defined by
$$h(\eta,t)= \chi(\eta) e^{i t(\phi(\la^{-1/2}\eta+\xi_B)-\phi(\xi_B)
-\inn {\la^{-1/2}\eta}{\nabla\phi(\xi_B)})},$$ is a Fourier
multiplier of $L^p$,
 uniformly in
$|t|\le \la$. However this follows since $\partial_\eta^\alpha
h(\eta,t) =O(1)$ for $|t|\le \la$ as  one  can easily check.
\end{proof}

\begin{proof}[Proof of Proposition \ref{ellipticthm}]
By a partition of unity and a compactness argument  it suffices to
show that for every $\xi_0\in \cU$ there is  a
 neighborhood $\cU(\xi_0)$ so that the statement of the theorem holds
with  $\chi$ replaced by   $\chi_0\in \coi $  supported in
$\cU(\xi_0)$. Now let $\mathcal{H}$ be the (symmetric) positive
definite squareroot of $\phi''(\xi_0)$ and let
$$\psi(\eta)= \eps_1^{-2} \left(\phi(\xi_0+\eps_1
\mathcal{H}^{-1}\eta)-\phi(\xi_0)- \eps_1
\inn{\mathcal{H}^{-1}\eta}{\nabla \phi(\xi_0)}\right).$$ Then it
suffices to show that $S^\psi$ (defined with the amplitude
$\chi(\xi_0+\eps_1 \mathcal{H}^{-1}\eta)$) satisfies the asserted
estimates, with a dependence on $\eps_1$. If $\eps_1$ is chosen
sufficiently small then we have reduced matters to a phase function
in $\Phi_\ellip(\eps, N,A)$ with parameters for which Tao's theorem
and therefore Lemma \ref{smallballsii} applies.

We now return to our original notation and work with a phase
function $\phi$ but  assume  now that  $\phi\in
\Phi_\ellip(\eps,N,A)$; we may also assume that the amplitude
function $\chi$ is smooth and supported in $[-(2d)^{-10},
(2d)^{-10}]^{-d}$. We  make a decomposition  of the product $Sf
\,Sg$ in terms of bilinear operators, localizing the frequency
variables in terms of nearness to the diagonal in
$(\xi,\eta)$-space; this is similar to arguments in
 \cite{lee-BR}, \cite{se-bil} and \cite{tavave}.

Let $\chi_{0}$ be a radial $\coi(\bbR^d)$ function  so that
$\chi_0(\omega)=1$ for $|\omega|\le 8d^{1/2}$ and so that
 $\supp \chi_0$ is contained in $\{\omega:|\omega|<16d^{1/2}\}$. Fix $\la>1$ and set
\begin{align*}
\Theta_{0}(\xi,\eta)&= \chi_0(\la^{1/2} (\xi-\eta))
\\
\Theta_j(\xi,\eta)&= \chi_0(\la^{1/2}2^{-j} (\xi-\eta))-
\chi_0(2\la^{1/2}2^{-j}(\xi-\eta)), \quad j\ge 1,
\end{align*}
so that $\Theta_0$ is supported where $|\xi-\eta|\le 16d^{1/2}
\la^{-1/2} $ and, $\Theta_j$  is supported in the region
$$4d^{1/2}2^{j} \la^{-1/2}\le |\xi-\eta|\le 16d^{1/2}2^{j}\la^{-1/2}.$$ We
may then decompose
$$
Sf \,Sg\,=\, \sum_{j\ge 0} \cB_j[f,g]
$$
where
$$
\cB_{j}[f,g](x,t)= \frac{1}{(2\pi)^{2d}} \iint
e^{i\inn{x}{\xi+\eta}} e^{it (\phi(\xi)+\phi(\eta))}
\Theta_j(\xi,\eta) \widehat f(\xi) \widehat g(\eta)  d\xi d\eta
$$
Only values of $j\ge 0$ with $2^j \le \la^{1/2}$ will be relevant,
as otherwise $\cB_j$ is identically zero. We will prove the estimate
\begin{equation}\label{claimbj}
\big\|\cB_j[f,g]\big\|_{p/2} \lc
\begin{cases}
2^{4j(\frac{d}{2}-\frac{d+1}p)}\la^{\frac 2p}\|f\|_p\|g\|_p, \quad &
\frac{2(d+3)}{d+1}< p\le 4,
\\
 2^{j(d-\frac
{4}{p})}\la^{\frac d2-\frac{2(d-1)}{p}} \|f\|_p\|g\|_p, \quad &
\quad \ \ 4<p<\infty,
\end{cases}
\end{equation}
and use this to bound
$$
\|Sf\|_{L^p(\bbR^d\times[0,\la])}=
\|(Sf)^2\|_{L^{p/2}(\bbR^d\times[0,\la])}^{1/2} \le \Big(\sum_{0\le
j \le \log_2(\la^{1/2})} \|\cB_j[f,f]\|_{p/2} \Big)^{1/2},
$$
and then sum a geometric series.

In order to prove \eqref{claimbj}, we decompose $\cB_j$ into pieces
on which we may apply Lemma~\ref{smallballsii}. Let $\vth\in
\coi(\bbR^d)$ a function supported in $[-3/5,3/5]^d$, equal to 1 on
$[-2/5,2/5]^d$,
 and satisfying  $$\sum_{n\in \bbZ^d} \vth(\xi-n)=1$$ for all $\xi\in \bbR^d$.
For  $j\ge 0$, $n\in \bbZ^d$, define
$$\beta_{j,n}(\xi)=\vth (\la^{1/2} 2^{-j}\xi -n)$$
and, for $(n,n')\in \bbZ^d\times\bbZ^d$,
\begin{equation*}
\vth_{j,n,n'}(\xi,\eta)= \Theta_j(\xi,\eta)  \beta_{j,n}(\xi)
\beta_{j,n'}(\eta).
\end{equation*}
Observe that $\beta_{j,n},\beta_{j,n'}$ are supported in cubes
$Q_{j,n}$, $Q_{j,n'}$ which have sidelengths  slightly larger than
$\la^{-1/2} 2^{j}$, and that are centered at the points
$\xi_{j,n}=\la^{-1/2} 2^{j} n$ and $\xi_{j,n'}=\la^{-1/2} 2^{j} n'$,
respectively.

Now let
\begin{align*}
\Delta_0 &= \{(n,n')\in \bbZ^d\times\bbZ^d\,:\, |n-n'|\le
18d^{1/2}\,\},
\\
\Delta &= \{(n,n')\in \bbZ^d\times\bbZ^d\,:\, 2d^{1/2}\le |n-n'|\le
18d^{1/2}\,\}.
\end{align*}
Then if $\vth_{0,n,n'}$ is not identically zero then we necessarily
have $(n,n')\in \Delta_0$ and if, for $j\ge 1$ the function
 $\vth_{j,n,n'}$ is not identically zero then we necessarily have
$(n,n')\in \Delta$. These statements follow by the definitions of
our cutoff functions. Moreover,
$$\dist (Q_{j,n}, Q_{j,n'}) \le 18d^{1/2}2^j\la^{-1/2} \quad\text{ if } (n,n')\in
\Delta_0,
$$
and
$$2^{-1}d^{1/2}2^j\la^{-1/2}\le \dist (Q_{j,n}, Q_{j,n'}) \le 18d^{1/2}2^j\la^{-1/2}\quad \text{ if } j\ge 1 \text{ and }(n,n')\in
\Delta.
$$

For the application of  Lemma \ref{smallballsii} it is convenient to
eliminate the cutoff $\Theta_j$ but still keep the separation of the
supports of $\beta_{j,n}$ and $\beta_{j,n'}$. Set, for $j\ge 1$,
$$
\widetilde \cB_{j}[f,g](x,t)= \frac{1}{(2\pi)^{2d}} \iint
e^{i\inn{x}{\xi+\eta}} e^{it (\phi(\xi)+\phi(\eta))} \sum_{n,n'\in
\Delta} \beta_{j,n}(\xi)\beta_{j,n'}(\eta) \widehat f(\xi) \widehat
g(\eta)  d\xi d\eta
$$
and define $\widetilde \cB_{0}[f,g]$ similarly by letting the
$(n,n')$ sum run over $\Delta_0$. The reduction of the estimate for
$\cB_j$ to the estimate for $\widetilde \cB_j$ is straightforward;
by an averaging argument. Indeed, let $\chi_1=
\chi_0-\chi_0(2\,\cdot\,)$ and use  the Fourier inversion formula
$$\Theta_{j} (\xi,\eta)
= \frac{1}{(2\pi)^{d}}\int \widehat\chi_1(y)
e^{i\la^{1/2}2^{-j}\inn{\xi-\eta}{y}} dy, \qquad j\ge 1;
$$
then
$$
\cB_j[f,g]= \frac{1}{(2\pi)^{d}} \int \widehat \chi_1(y)\widetilde
\cB_j[f_{-y}, g_{y}] dy
$$
where $f_{-y}(x)= f(x+\la^{1/2}2^{-j}y)$ and
$g_{y}(x)=g(x-\la^{1/2}2^{-j}y)$. A similar formula holds for $j=0$,
only then $\chi_1$ is replaced with $\chi_0$. Thus in order to
finish the argument it is enough to show that $ \|\widetilde \cB_j[
f,g]\|_{p/2}$ is dominated by the right hand side of
\eqref{claimbj}.

Define convolution operators $P_{j,n}$ by $\widehat {P_{j,n}  f}
=\beta_{j,n}\widehat f$. Note that for fixed $j$, each $\xi$ is
contained in only a bounded number of the sets $Q_{j,n}+Q_{j,n'}$.
This implies, by interpolation of $\ell^2(L^2) $ with trivial
$\ell^1(L^1)$ or $\ell^\infty(L^\infty)$ bounds that, for $j\ge 1$,
$p\ge 2$,
\begin{multline}\label{nn'orth}
\big\|\widetilde {\cB}_j[f,g]\big\|_{L^{p/2}(\bbR^d\times[0,\la])}
\\ \lc \max\{ 1, (\la^{1/2} 2^{-j})^{d(1-4/p)}\} \,
\Big(\sum_{n,n'\in \Delta} \big\|SP_{j,n} f\, SP_{j,n'} g
\big\|_{L^{p/2}(\bbR^d\times[0,\la])}^{p/2}\Big)^{2/p}.
\end{multline}
The
 analogous formula for $j=0$ holds if we replace $\Delta$ by $\Delta_0$.
Notice that for all $j$,
\begin{equation}\label{pjnest}
\Big(\sum_n\|P_{j,n} f\|_p^p\Big)^{1/p}
 \lc \|f\|_p, \quad\text{$p\ge 2$}.
\end{equation}

Now if $j=0$ we use Lemma \ref{smallballsi} to estimate
\begin{align*}
\|SP_{0,n} f(\cdot,t)\, SP_{0,n'} g(\cdot,t)\big\|_{L^{p/2}(\bbR^d)}
&\lc \|SP_{0,n} f(\cdot,t)\|_p\| SP_{0,n'} g(\cdot,t)\|_p\\
&\lc\|P_{0,n}f\|_p \|P_{0,n'}g\|_p;
\end{align*}
hence,  after integrating in $t$,
\begin{align*}
\big\|\widetilde
{\cB}_0[f,g]\big\|_{L^{p/2}(\bbR^d\times[0,\la])}\lc \max\{ 1,
\la^{d(1/2-2/p)}\} \la^{2/p} \Big( \sum_{n,n'\in \Delta_0}\|P_{0,n}
f\|_p^{p/2} \|P_{0,n'} g\|_p^{p/2}\Big)^{2/p}&
\\
\lc \max\{ 1, \la^{d(1/2-2/p)}\} \la^{2/p} \Big(\sum_n\|P_{0,n}
f\|_p^p\Big)^{1/p} \Big(\sum_{n'}\|P_{0,n'} g\|_p^p\Big)^{1/p}.&
\end{align*}
 The asserted bound for $j=0$ follows
from \eqref{pjnest}.

Next for $j>0$ we use Lemma \ref{smallballsii}, and thus the
assumption $p>2+\frac{4}{d+1}$,  and estimate
$$\big\|SP_{j,n} f\, SP_{j,n'} g\big\|_{L^{p/2}(\bbR^d\times[0,\la])}
\lc 2^{4j(\frac d2-\frac{d+1}p)} \la^{2/p}
 \|P_{j,n} f\|_p \|P_{j,n'} g\|_p.
$$
Therefore by \eqref{nn'orth}
\begin{multline*}
\Big\|\widetilde {\cB}_j[f,g]\Big\|_{L^{p/2}(\bbR^d\times[0,\la])} \\
\lc \max\{ 1, (\la^{1/2} 2^{-j})^{d(1-4/p)}\} 2^{4j(\frac
d2-\frac{d+1}p)} \la^{2/p} \Big(\sum_n\|P_{j,n} f\|_p^p\Big)^{1/p}
\Big(\sum_{n'}\|P_{j,n'} g\|_p^p\Big)^{1/p}
\end{multline*}
and  again the asserted bound for $\|\widetilde \cB_j[f,g]\|_{p/2}$
follows
 from \eqref{pjnest}.
\end{proof}

\section{Estimates for
$\exp(it (-\Delta)^{\alpha/2})$ }\label{localtoglobal}

We now prove the endpoint estimates of Theorems \ref{pop1} and
\ref{the}. First we remark that by various scaling and symmetry
arguments we may assume that $I=[0,1]$.

Consider $\chi_0,\chi\in \coi(\bbR)$ supported in $(-2,2)$ and
$(1/2,2)$, respectively, such that
$$
\chi_0+\sum_{k\ge1}\chi(2^{-k}\,\cdot\,)=1.
$$
We define the operators $T_k^\alpha\equiv T_k$ by
$$
\widehat{T_{0} f(\,\cdot\,,t )}(\xi) =\chi_0(|\xi|)
e^{it|\xi|^\alpha}\widehat f(\xi),
$$
$$\widehat{T_{k} f(\,\cdot\,,t )}(\xi) =\chi(2^{-k}|\xi|) e^{it|\xi|^\alpha}\widehat
f(\xi),\quad k\ge 1,
$$
so that $U^\alpha_t=\sum_{k\ge0}T_k(\cdot,t)$.

Our main result is the following  inequality for vector-valued functions $\{f_k\}_{k=0}^\infty \in \ell^p(L^p)$.
\begin{theorem} \label{trliz}
Let $p\in(2+\frac{4}{d+1},\infty)$, $\alpha\neq1$, $d=1$ or
$\alpha>1$, $d\ge2$ and $\beta =\alpha d (\frac{1}{2}-
\frac{1}{p})-\frac{\alpha}{p}$. Then
\begin{equation}\label{trlizest}
\Big\|\, \sum_{k\ge0}  \Big(\int_0^1|2^{-k\beta} T_{k} f_{k}(\cdot,
t)|^pdt\Big)^{1/p}
 \,\Big\|_{L^p(\R^d)} \lc
\Big(\sum_{k\ge0}
\|f_{k}\|_{p}^p\Big)^{1/p}.
%\quad \frac
%\beta\alpha= d \Big(\frac 12- \frac 1p\Big)-\frac1p.
\end{equation}
\end{theorem}
The proof will be given in  \S\ref{mainproof}.
We now  discuss the implications to Theorem \ref{pop1} and~\ref{the}, in
fact strengthened versions involving Triebel-Lizorkin spaces
$F^{p}_{\alpha,q}$ and Besov spaces $B^p_{\alpha,q}$.
Here the norms on these spaces are given by the $L^p(\ell^q)$
and $\ell^q(L^p)$ norms (resp.) of the sequence $\{ 2^{k\alpha}
L_k f\}_{k=0}^\infty$,
with the usual inhomogeneous dyadic frequency composition
$I=\sum_{k\ge 0} L_k$.
See \cite{triebel}. The
following corollary is an immediate consequence of Theorem \ref{trliz}, by
Minkowski's
inequality and  Fubini's theorem.

\begin{corollary} \label{smcor}
Let $p$, $\alpha$, $\beta$  be as in Theorem  \ref{trliz}.
Then
\begin{equation*}\label{smcorest}
\Big(\int_0^1 \big \|U^\alpha_{t} f \big\|_{F^p_{0,1}(\bbR^d)}^p
dt\Big)^{1/p}
\lc \|f\|_{B^p_{\beta,p}(\bbR^d)}.
\end{equation*}
\end{corollary}

This implies Theorem \ref{the} since
for $p\ge 2$  the space
$B^p_{\beta,p}\equiv F^p_{\beta,p}$ contains the Sobolev space $L^p_\beta\equiv F^{p}_{\beta,2}$, via the embedding $\ell^2\hookrightarrow\ell^p$ followed by the Littlewood-Paley inequality, and by the same reasoning $F^p_{0,1}$        is imbedded in $L^p\equiv F^{p}_{0,2}$.
We remark that a similar sharp  inequality for the wave equation
is proved in \cite{nase},
in sufficiently high dimensions.

Another consequence of Theorem \ref{trliz} is
\begin{corollary} \label{maxcor}
Let $p$, $\alpha$,  be as in Theorem  \ref{trliz}.
Let $t\mapsto \vartheta(t) $ be   smooth  and  compactly supported. Then
\begin{equation}\label{maxcorest}
\Big\| \, \big\| \vartheta(\cdot) U^\alpha_{(\cdot)} g
\big\|_{B^p_{1/p,1}(\bbR)} \Big\|_{L^p(\bbR^d)}
\lc \|g\|_{B^p_{\gamma,p}(\bbR^d)}, \quad\gamma= \alpha d(1/2-1/p).
\end{equation}
\end{corollary}
 Theorem  \ref{pop1} is an immediate consequence of Corollary
 \ref{maxcor} since   the Besov space $B^p_{1/p,1}(\bbR)$ is continuously
embedded  in the space $C^0$ of continuous bounded functions
which vanish at infinity.

To see how Corollary \ref{maxcor}  follows from Theorem \ref{trliz}
 we introduce dyadic frequency cutoffs in the $t$ variable. We
 decompose  the identity as $I=\sum_{j=0}\cL_j$ where
$\widehat {\cL_j f} (\tau)= \widetilde \chi_j (\tau)\widehat f(\tau)$ where
$\widetilde \chi_j =\widetilde \chi (2^{-j}|\cdot|)$ for $j\ge 1$,
 with a suitable $\widetilde \chi\in C^\infty_0$ supported in
 $(1/2,2)$ and $\widetilde \chi_0$ is smooth and vanishes for
 $|\tau|\ge 2$. Now we apply  $L_j$ to
$\vartheta T_k g$. If  $2^{j-\alpha k} \notin (2^{-10}, 2^{10})$,
then we  apply an integration by parts in $s$  to terms
of the form $$\iint \chi(2^{-j}|\tau|) \chi(2^{-k}|\xi|) \widehat
 g(\xi) e^{i(\inn{x}{\xi}+t\tau)} \int \vartheta(s) e^{is
 (|\xi|^\alpha-\tau)} ds \, d\xi d\tau.$$ One finds that for this range
the contribution of $\cL_j [\vartheta T_k g]$
is negligible;
%if $2^{j-ak} \notin (2^{-10}, 2^{10})$,
namely
 $$\Big(\int_{\R} \int_{\bbR^d}  | \cL_j [\vartheta T_k g ] (x,s)|^p
dx ds\Big)^{1/p} \lc C_N \min \{2^{-\alpha kN}, 2^{-jN}\} \|g\|_p \, \text{ if }
2^{j-\alpha k} \notin (2^{-10}, 2^{10}).
$$
Thus a localization in $\xi$ where $|\xi|\approx 2^k$ corresponds to a localization in $\tau$ where $|\tau |\approx 2^{k\alpha}$.
We combine this with Theorem \ref{trliz} applied to
$f_k= 2^{k\beta+ k/p} \cF^{-1} [\chi (2^{-k} |\cdot|) \widehat g]$
and obtain
\begin{equation*}
\Big\| \sum_{j\ge 0} 2^{j/p} \big \|\cL_j [\vartheta
  U^\alpha_{(\cdot)} g]\|_{L^p(\bbR), dt} \Big\|_{L^p(\bbR^d)}
\lc \Big(\sum_{k\ge 0} 2^{k \gamma p}\big\|\cF^{-1}[ \chi_k \widehat
  g]\big\|_{L^p(\bbR^d)}^p\Big)^{1/p}
\end{equation*}
which is \eqref{maxcorest}.

% The
%general case follows using
%$\|\cdot\|^q_q\le\|\cdot\|^{q-p}_\infty\|\cdot\|^p_p$.

%Let $\widetilde \chi\in \coi(\bbR)$ be supported in $(1/3,3)$ so
%that $\widetilde \chi =1$ on $\supp \chi$, and define $f_{k}$ by
%$$\widehat{f_{k}}(\xi)=\widetilde \chi(2^{-k}|\xi|) \widehat f(\xi).$$

% When $p\le q$, we have
%the Sobolev embedding inequality
%%We provide a simple proof for convenience.
%\begin{equation}\label{sob}
%\|T_{k}f(x,\cdot)\|_{L_t^q[0,1]}\lc 2^{\alpha k(\frac1p-\frac1q)}
%\|T_{k}f(x,\cdot)\|_{L_t^p[0,1]}.
%\end{equation}
%For $q=\infty$, this is proven  by applying the Fundamental Theorem
%of Calculus to $T_{k}f(x,\cdot)^p$ (see \cite{caruve}). The
%general case follows using
%$\|\cdot\|^q_q\le\|\cdot\|^{q-p}_\infty\|\cdot\|^p_p$.

%We combine these estimates in the following proposition from which
%it is easy to deduce Theorems~\ref{pop1} and \ref{the}. To obtain
%Theorem~\ref{the}, we applying Minkowski's inequality in the
%temporal norm on the left hand side, and on the right hand side we
%change the order of the sum and the integral, then use the embedding
%$\ell^2\hookrightarrow \ell^p$ combined with the Littlewood-Paley
%inequality. Theorem~\ref{pop1} follows in the same way, with an
%application of \eqref{sob} with $q=\infty$ to begin with.

\section{Proof of Theorem \ref{trliz}}
\label{mainproof}
The localization of the multiplier near the origin $T_0$ is easily
handled as
$$\|\cF^{-1}[\chi_0(|\cdot|) e^{it|\cdot|^\alpha}]\|_{L^1}\le C$$
uniformly for $t\in [0,1]$. To see this, since
$\cF^{-1}[\chi_0(|\cdot|)]\in L^1$, it suffices to show that for
$\phi$ supported in $(1/2,2)$, the $L^1$ norm of
 $ \cF^{-1}[\chi_0 (e^{it|\cdot|^\alpha} -1)\phi(2^k|\cdot|)]$ is
$O(2^{-\alpha k})$ for $k\ge 0$. But by scaling this follows from
showing that the $L^1$ norm of $ \cF^{-1}[\chi_0(2^{-k}\cdot)
(e^{it2^{-\alpha k}|\cdot|^\alpha} -1)\phi(|\cdot|)]$ is
$O(2^{-\alpha k})$ which follows from the standard Bernstein
criterion.

Now, by scaling and  Proposition \ref{ellipticthm} with $\la\approx
2^{\alpha k}$, $\cU=\{\xi:1/2<|\xi|<2\}$ and
$\phi(\xi)=|\xi|^\alpha$,
 we have already proven the estimates
\begin{equation}\label{freq}
\|T_{k} f\|_{L^{p}(\bbR^d\times[0,1])}\lc
2^{k\beta}\|f\|_{L^{p}(\bbR^d)},\quad \beta\ge \beta(p):=\alpha d
\Big(\frac 12- \frac 1p\Big)-\frac\alpha p
\end{equation}
for $k>0$ and $p>2+\frac{4}{d+1}$.

 It suffices thus to show  that if
(\ref{freq}) holds for all $k>0$ and all $p>q$, then
(\ref{trlizest}) holds for all $p\in(q,\infty)$. Due to our
 restriction on \eqref{freq} we let $q=2+\frac{4}{d+1}$ and fix
 $2+\frac{4}{d+1}<r<p$.
%The argument also works when $\alpha=1$, but again, we have not
% proven
%\eqref{freq} for $\alpha=1$.
We can make the
additional assumption that the $k$ sum on the left hand side is
extended over a finite set (with the constant in the inequality
independent of this assumption);  the general case then follows by
the monotone convergence theorem.
%\footnote{\texttt{do we use the
%fact that the sum is finite? Can/should we delete this sentence?}}

For later reference we state a
Sobolev inequality which is proved linking frequency
decompositions in $\xi$ and $\tau$   and Young's inequality
(just as in the argument used in \S\ref{localtoglobal}  to deduce
 Corollary \ref{maxcor} from Theorem \ref{trliz}). Namely
 \begin{equation}\label{sob}
\big\|\|T_{k}f\|_{L_t^p[0,1]}\big\|_{L_x^r}\lc 2^{\alpha k(\frac1r-\frac1p)}
\big\|\|T_{k}f\|_{L_t^r[0,1]}\big\|_{L_x^r}.
\end{equation}
holds
for $r\le p\le \infty$ (including the endpoint).
Alternatively one can also
apply the fundamental theorem
of calculus to $|T_{k}f(x,\cdot)|^r$ (see e.g. \cite{st}) to get
\eqref{sob} for $p=\infty$ and the general inequality follows by
convexity.
%log-convexity of $L^q$ norms.

The main ingredient in the proof of  \eqref{trlizest} (besides \eqref{freq})
  will be the Fefferman-Stein sharp function
\cite{fest} and their inequality $$\|F\|_p\lc \|F^\#\|_p,$$ where
$p\in(1,\infty)$ and \textit{a priori} $F\in L^p$.
We apply this to $\sum_{k>0} 2^{-k \beta(p)}\|T_{k}
f_{k}(x,\cdot)\|_{L^p_t[0,1]}   $ and by
\eqref{freq} this function  is \textit{a priori} in $L^p$ as the sum in $k$
is assumed to be finite.
Thus  it
will suffice to prove that
$$
%\begin{multline*}
\Big\|\sup_{x\in Q} \intslash_Q \Big| \sum_{k>0}2^{-k \beta(p)} \|T_{k}
f_{k}(y,\cdot)\|_{L^p_t[0,1]}
- \intslash_Q \sum_{k>0}2^{-k \beta(p)}
 \|T_{k}
f_{k}(z,\cdot)\|_{L^p_t[0,1]}\,dz\Big| dy \Big\|_{L^p_x} \\
 %\lc\Big(\sum_{k>0}\|f_{k}\|_p^p\Big)^{1/p}\!\!,
$$
%\end{multline*}
is dominated by $C(\sum_{k>0}\|f_{k}\|_p^p)^{1/p}$. Here
 the supremum is taken over all  cubes containing $x$, and the
slashed integral
 denotes the average
$|Q|^{-1}\int_Q$.
%Setting $$
%\beta(p)=\alpha d\left(\frac{1}{2}-\frac{1}{p}\right)-\frac \alpha
%p,$$
By the triangle inequality the previous bound
follows from
\begin{equation*}\label{sharp}
\Big\|\sup_{x\in Q} \intslash_Q \sum_{k>0}
\intslash_Q2^{-k\beta(p)}\|T_{k} f_{k}(y,\cdot)  -  T_{k}
f_{k}(z,\cdot)\|_{L^p_t[0,1]}dz dy \Big\|_{L^p_x} \lc
\Big(\sum_{k}\|f_{k}\|_p^p\Big)^{1/p}.
\end{equation*}

Denoting the sidelength of $Q$ by $\ell(Q)$, we observe that, by
Minkowski's inequality, this would follow from the inequalities
\begin{equation}\label{sing}
\Big\|\sup_{x\in Q} \intslash_Q \sum_{2^k \ell(Q)\le 1} \intslash_Q
2^{-k\beta(p)}\|T_{k} f_{k}(y,\cdot)  -  T_{k}
f_{k}(z,\cdot)\|_{L^p_t[0,1]}dz dy \Big\|_{L^p_x}  \lc
\Big(\sum_k\|f_{k}\|_p^p\Big)^{1/p},
\end{equation}
\begin{equation}\label{dispk1}
\Big\|\sup_{x\in Q} \intslash_Q \sum_{2^k \ell(Q)> 2^{\alpha k}}
2^{-k\beta(p)} \| T_{k} f_{k}(y,\cdot)
\|_{L^p_t[0,1]}\,dy\Big\|_{L^p_x} \lc
\Big(\sum_k\|f_{k}\|_p^p\Big)^{1/p}.
\end{equation}
and
\begin{equation}\label{disp}
\Big\|\sup_{x\in Q} \intslash_Q \sum_{\,2^{\alpha k}\ge2^k \ell(Q)>
1} 2^{-k\beta(p)} \| T_{k} f_{k}(y,\cdot)
\|_{L^p_t[0,1]}\,dy\Big\|_{L^p_x} \lc
\Big(\sum_k\|f_{k}\|_p^p\Big)^{1/p}.
\end{equation}
First we handle \eqref{sing} and \eqref{dispk1} by standard
estimates and then prove the more interesting
inequality~\eqref{disp}.

\subsection*{{\it Proof of \eqref{sing}}}
It is enough to consider cubes $Q$ of diameter $\approx 2^{j}$ with
$x,y,z\in Q$ and $j+k\le0$. Let $H_k=\cF^{-1}[\widetilde
\chi(2^{-k}|\cdot|)]$, where $\widetilde{\chi}$ is smooth, equal to
one on $(1/2,2)$, and supported in $(1/3,3)$. Then
$$|\nabla H_k(w)|\lc  2^{k}   \frac{2^{kd}}{(1+2^k |w|)^{2N}}$$
with large $N\ge 10d$. Thus
\begin{align*}T_{k} f_{k}(y,t)- T_{k} f_{k}(z,t)&=\int\Big[H_k(y-w)- H_k(z-w)\Big] T_{k} f_{k}(w,t)
dw\\
&= \int \int_0^1 \biginn {(y-z)}{\nabla H_k(z+s(y-z)-w) T_{k}
f_{k}(w,t) }\,ds\,dw
\end{align*} which is controlled by a constant multiple of
$$
2^{j+k} \int \frac{2^{kd}}{(1+2^k|x-w|)^N} |T_{k} f_{k}(w,t)| dw.
$$
Thus, using the embedding $\ell^p\hookrightarrow \ell^\infty$, the
right hand side of \eqref{sing} is bounded by
\begin{align*}
&\,\,\,\,\,\Big\|\Big(\sum_j\Big| \sum_{0<k\le -j} \Big\|2^{j+k}
 \int \frac{2^{kd}}{(1+2^k|\cdot -w|)^N} 2^{-k\beta(p)}|
T_{k} f_{k}(w,\cdot)|
dw\Big\|_{L_t^p[0,1]}\Big|^p\Big)^{1/p}\Big\|_{L^p_x}
\\
&\lc \, \sum_{n\ge 0} 2^{-n}\! \left(\sum_{j<-n}  \Big\|
 \int \!\!\frac{2^{-(n+j)(d-\beta(p))}}{(1+2^{-(n+j)}|\cdot -w|)^N}|
 % 2^{(n+j)\beta(p)}|
T_{-(n+j)} f_{-(n+j)}(w,\cdot)|
dw\Big\|^p_{L^p(\R^d\times[0,1])}\right)^{1/p}
\\
&\lc \, \sum_{n\ge 0} 2^{-n} \left(\sum_{j<-n}
\big\|2^{(n+j)\beta(p)}
 T_{-(n+j)} f_{-(n+j)}
\big\|^p_{L^p(\R^d\times[0,1])}\right)^{1/p}.
\end{align*}
By \eqref{freq}  the last expression is dominated by
a constant times
\begin{align*}
&\sum_{n\ge 0} 2^{-n} \Big(\sum_{j<-n } \big\|
 f_{-(n+j)}
\big\|_p^p \Big)^{1/p} \lc \Big(\sum_k\|f_{k}\|_p^p\Big)^{1/p}
\end{align*}
and \eqref{sing} is proved.

\subsection*{{\it Proof of \eqref{dispk1}}}
For a fixed $t$, the operator $T_k$ has convolution kernel $K^t_k$
 given by
\begin{equation*}
K^t_k(x)= \frac{2^{kd}}{(2\pi)^{d}} \int_{\bbR^d} \chi(|\xi|)
e^{i(2^k \inn x\xi+2^{\alpha k}t|\xi|^\alpha)} d\xi.
\end{equation*}
Let $C(\alpha)=1$ if $\alpha\in(0,1)$ and let
$C(\alpha)= \alpha 2^{\alpha -1}$ if $\alpha\in(1,\infty)$, and
define
$$\fB_k(\alpha)=\{ x\,:\, |x|\le  4C(\alpha)
2^{k(\alpha-1)}\}.
$$ Integration by parts yields favorable bounds
in the complement of this ball. Observe that
\begin{equation*}
\big|\nabla_\xi \big( 2^k \inn x\xi+2^{\alpha
k}t|\xi|^\alpha\big)\big| \ge c_\alpha 2^k |x| \text{ if } x\notin
\fB_k(\alpha), \quad t\in [0,1],
\end{equation*}
and we obtain
\begin{equation} \label{smallness}
|K^t_k(x)|\le C_{N} 2^{kd}(1+2^k|x|)^{-N} \text{ if } x\notin
\fB_k(\alpha), \quad t\in [0,1].
\end{equation}
Consequently the main contribution of $K^t_k(x)$ comes when $|x|\le
4C(\alpha) 2^{k(\alpha-1)}$.

We prove the estimate \eqref{dispk1} by interpolation between
\begin{equation*}
\Big\|\sup_{x\in Q} \intslash_Q
 \sum_{2^k
\ell(Q)> 2^{\alpha k}} 2^{-k\beta(p)} \| T_{k} f_{k}(y,\cdot)
\|_{L^p_t[0,1]} dy\Big\|_\infty \lc \sup_k\|f_k\|_\infty
\end{equation*}
and
\begin{equation*}
\Big\|\sup_{x\in Q} \intslash_Q
 \sum_{2^k
\ell(Q)> 2^{\alpha k}} 2^{-k\beta(p)} \| T_{k} f_{k}(y,\cdot)
\|_{L^p_t[0,1]} dy\Big\|_r \lc \Big(\sum_k\|f_k\|_r^r\Big)^{1/r},
\end{equation*}
where $2+\frac{4}{d+1}<r<p$.

Now, as $\beta(p)>\beta(r)+\alpha(\frac{1}{r}-\frac{1}{p})$, the
$L^r$ bound is proven by applying
 H\"older in $k$, followed by the inequality
\begin{equation*}
\big\|\sup_{x\in Q} \intslash_Q \Big(\sum_k 2^{-k\big(\beta(r)+\alpha(\frac{1}{r}-\frac{1}{p})\big)r}\| T_{k}
f_{k}(y,\cdot) \|_{L^p_t[0,1]}^r\Big)^{1/r}dy \big\|_r
 \lc \Big(\sum_k\|f_k\|_r^r\Big)^{1/r}.
\end{equation*}
This is a consequence of the $L^r$--boundedness of  the
Hardy--Littlewood maximal operator, the interchange of the spatial integral and the sum, an application of \eqref{sob}, followed by Fubini and the estimate~\eqref{freq} (for the admissible exponent $r>2+4/(d+1)$).

To prove the $L^\infty$ bound, we let $Q^*$ be a  cube with the same
center as $Q$ satisfying $\ell(Q^*)= 10d C(\alpha) \ell(Q)$. By
Minkowski's inequality it will suffice to prove that
\begin{equation}\label{first}
 \intslash_Q
 \sum_{2^k
\ell(Q)> 2^{\alpha k}}
 2^{-k\beta(p)} \| T_{k} [f_{k}\chi_{Q^*}](y,\cdot)
\|_{L^p_t[0,1]}
dy\lc \sup_k\|f_k\|_\infty
\end{equation}
and
\begin{equation}\label{second}
 \intslash_Q
 \sum_{2^k
\ell(Q)> 2^{\alpha k}}
 2^{-k\beta(p)} \| T_{k} [f_{k} \chi_{\bbR^d\setminus Q^*}](y,\cdot)
\|_{L^p_t[0,1]}
dy\lc \sup_k\|f_k\|_\infty
\end{equation}
uniformly in $Q$.

To prove \eqref{first}, again we apply  H\"older a
number of times and \eqref{sob};
\begin{align*}
&\intslash_Q
 \sum_{k}
 2^{-k\beta(p)} \| T_{k} [f_{k}\chi_{Q^*}](y,\cdot)
\|_{L^p_t[0,1]}
dy\\
&\lc  |Q|^{-1/r} \sum_k2^{-k(\beta(p)-\alpha(\frac{1}{r}-\frac{1}{p}))}\Big(\int \|
T_{k} [f_{k}\chi_{Q^*}](y,\cdot) \|_{L^r_t[0,1]}^rdy\Big)^{1/r}
\\
&\lc \sup_k |Q|^{-1/r}
2^{-k\beta(r)}\Big(\int \|
T_{k} [f_{k}\chi_{Q^*}](y,\cdot) \|_{L^r_t[0,1]}^rdy\Big)^{1/r}
\\
&\lc \sup_k |Q|^{-1/r} \Big(\int  |f_k \chi_{ Q^*}|^rdx\Big)^{1/r}
\lc \sup_k\|f_k\|_\infty,
\end{align*}
where the third
inequality holds again by  the $L^r$ version of \eqref{freq}.

For \eqref{second}, we note  that
as $\ell(Q)>2^{k(\alpha-1)}$, and the function is supported in
the complement of $Q^*$ we can use the rapid decay in formula \eqref{smallness}.
We have that
\begin{align*}
&\intslash_Q
 \sum_{2^k
\ell(Q)> 2^{\alpha k}}
 2^{-k\beta(p)} \| T_{k} [f_{k}\chi_{\R^d\setminus Q^*}](y,\cdot)
\|_{L^p_t[0,1]}
dy
\\
\lc&  \sup_{k} \intslash_Q\left\|\int
\frac{2^{kd}}{(1+2^k|y-z|)^{2d}}
|f_k(z)| dz\right\|_{L^p_t[0,1]}dy\\
\lc& \sup_{k} \left\|\int
\frac{2^{kd}}{(1+2^k|\cdot-z|)^{2d}}
|f_k(z)| dz\right\|_{\infty}\lc \sup_{k} \|f_k\|_{\infty}.
\end{align*}
This concludes the proof of~\eqref{dispk1}

\subsection*{{\it Proof of \eqref{disp}}}\label{main}

We let $\zeta_j(x) =
(d2^{j})^{-d}$ if $|x|\le d 2^j$ and $\zeta_j(x)=0$ if $|x|\ge
d2^j$. Replacing cubes by dyadic balls we see that \eqref {disp}
follows from
\begin{equation}\label{displp}
\Big\|\sup_j \,\zeta_j* \sum_{k+j>0\atop{(\alpha-1)k\ge j}}
2^{-k\beta(p)} \| T_{k} f_{k}\|_{L^p_t[0,1]} \Big\|_{L^p_x} \lc
\flplp.
\end{equation}

Now, for fixed $k$ we cover $\Bbb R^d$ by  a grid $\cR_k^{\alpha
-1}$ consisting of cubes of sidelength $2^{k(\alpha-1)}$. For each
$R\in \cR_k^{\alpha -1}$ let $R^*$ be the cube with same center as
$R$ and sidelength $C(\alpha) 2^{k(\alpha-1) +10d}$
where $C(\alpha)$ is as in the proof of \eqref{dispk1}

For $R\in \cR_k^{\alpha -1}$ we let $f_{k}^R= \chi_R f_{k}$. We may
then split the left hand side of \eqref{displp} as $I +II$ where

\begin{equation*}
%\label{Idef}
I=\Big\|\sup _j\,\zeta_j* \Big[\sum_{k+j>0\atop{(\alpha-1)k\ge j}}
 2^{-k\beta(p)} \| \sum_{R\in
\cR_{k}^{\alpha -1}}\chi_{R^*} T_{k} f_{k}^R \|_{L^p_t[0,1]}\Big]
\Big\|_{L^p_x}
\end{equation*}
and $II$ is the analogous expression where $ \chi_{R^*}$ is replaced with
$\chi_{\bbR^d\setminus R^*}$.

By Hardy--Littlewood, Minkowski, Fubini, (\ref{smallness}), and
Young's inequality,
 we dominate
\begin{align*}
II&\lc   \sum_{k\ge0} 2^{-k\beta(p)}\Big\| \sum_{R\in
\cR_{k}^{\alpha -1}} \chi_{\bbR^d\setminus R^*}\, T_{k} f_{k}^R
\Big\|_{L^p(\R^d\times[0,1])}
\\
&\lc  \sum_{k\ge0} 2^{-k\beta(p)}\Big(\int_0^1\int \Big[\int
\frac{2^{kd}}{(1+2^k|x-y|)^{2d}}\sum_{R\in \cR_{k}^{\alpha -1}}
|f_{k}^R (y)| dy\Big]^p dxdt\Big)^{1/p}
\\&\lc  \sum_{k\ge0} 2^{-k\beta(p)} \Big\|\sum_{R\in \cR_{k}^{\alpha -1}}
f_{k}^R \,\Big \|_p \lc\sup_k\|f_{k}\|_p \lc\flplp.
\end{align*}

Concerning the main term $I$ we use the  imbedding
$\ell^p\hookrightarrow \ell^\infty$,  interchange a sum and an
integral, and apply Minkowski's inequality,  so that
$$I\lc
\Big(\sum_j \Big\|\zeta_j* \Big[\sum_{k+j>0\atop{(\alpha-1)k\ge j}}
 2^{-k\beta(p)} \sum_{R\in \cR_{k}^{\alpha -1}} \chi_{R^*}\|
T_{k} f_{k}^R
\|_{L^p_t[0,1]}\Big] \Big\|^p_{L^p_x}\Big)^{1/p}.$$
Now for $R\in \cR^{\alpha-1}_k$, $R^*$  has sidelength greater than $2^j$,
so for fixed
$k$ the functions $\zeta_j*\chi_{R^*}$
have bounded overlap, uniformly in $k$. Setting $n=k+j>0$ and applying
Minkowski's inequality, we get $$I\lc  \sum_{n>0} I_n$$ where
\begin{equation*}
I_n=\Big(\sum_{j<n}\sum_{R\in \cR_{n-j}^{\alpha -1}}2^{-(n-j)\beta(p)p}\Big\| \zeta_j*
  \|
T_{n-j} f_{n-j}^R \|_{L^p_t[0,1]} \Big\|^p_{L^p_x}\Big)^{1/p}.
\end{equation*}
As before choose $r$ so that  $2+\frac{4}{d+1}<r<p$. It will suffice to show  that
\begin{equation}\label{Inest}
I_n \lc  2^{-nd (\frac{1}{r}-\frac{1}{p})}  \flplp.
\end{equation}

Observe that by Young's inequality the convolution with $\zeta_j$
maps $L^r(\bbR^d)$ to $L^p(\bbR^d)$ with operator norm $O(
2^{-jd(1/r-1/p)})$. Moreover by (\ref{sob}) we have
\begin{equation*}
\Big\|\|T_{n-j} f_{n-j}^R\|_{L_t^p[0,1]}\Big\|_{L_x^r}\lc 2^{(n-j)\alpha (\frac1r-\frac1p)}\Big\|
\|T_{n-j} f_{n-j}^R\|_{L_t^r[0,1]}\Big\|_{L_x^r}.
\end{equation*}
Thus we can bound
\begin{equation*}
I_n\lc \Big(\sum_j 2^{-jd(\frac1r-\frac1p)p} 2^{(n-j)\alpha
(\frac1r-\frac1p)p} 2^{-(n-j)\beta(p) p} \sum_{R\in
\cR_{n-j}^{\alpha -1}} \big\|  T_{n-j}
f_{n-j}^{R}\big\|_{L^r(\R^d\times[0,1])}^p \Big)^{\frac 1p}
\end{equation*}
which, by \eqref{freq}, is
\begin{equation*}
\lc \Big(\sum_j 2^{-jd(\frac1r-\frac 1p)p} 2^{(n-j)\alpha
(\frac1r-\frac 1p)p} 2^{-(n-j)\beta(p) p} \sum_{R\in
\cR_{n-j}^{\alpha -1}} 2^{(n-j)\beta(r)p}\big\|f_{n-j}^{R}\big\|_r^p
\Big)^{\frac 1p}.
\end{equation*}
Since $f_{n-j}^{R}$ is supported on the  cube $R$  of size
$2^{(n-j)(\alpha-1) d}$ we see by H\"older's inequality that the
last displayed expression is dominated by a constant times
\begin{equation*}
 \Big(\sum_j 2^{-jd(\frac 1r-\frac 1p)p} 2^{(n-j)\alpha (\frac 1r-\frac 1p)p}
2^{-(n-j)\beta(p) p} 2^{(n-j)\beta(r)p} 2^{(n-j)(\alpha-1) d (\frac
1r-\frac 1p)p} \!\!\!\!\!\sum_{R\in \cR_{n-j}^{\alpha -1}}
\!\!\big\|f_{n-j}^{R}\big\|_p^p \Big)^{\frac 1p}.
\end{equation*}
Now this simplifies, after
summation in $R$, to
\begin{align*}   I_n \lc  2^{-nd(\frac 1r-\frac 1p)}
\Big(\sum_j \|f_{n-j}\|_p^p\Big)^{\frac 1p}\le C  2^{-nd(\frac
1r-\frac 1p)} \flplp.
\end{align*}
This finishes the proof of \eqref{Inest} and thereby \eqref{disp}
and concludes the proof  of Theorem~\ref{trliz}. \qed

%\vspace{2em}

\medskip {\it Acknowledgement.} The first author  thanks  Gustavo Garrig\'os,
Manuel Portilheiro and Ana Vargas for helpful conversations, and
Luis Vega for bringing the $L^p$--conjecture for the Schr\"odinger
maximal operator to his attention. The authors also thank Jong-Guk
Bak and Sanghyuk Lee for a correction concerning the case $\alpha<1$.

\end{document}